\documentclass[a4paper,12pt]{amsart}

\usepackage{euscript,eufrak,verbatim}
\usepackage[usenames]{color}

\newtheorem{theorem}{Theorem}[section]
\newtheorem{lemma}[theorem]{Lemma}
\newtheorem{proposition}[theorem]{Proposition}

\begin{document}

\title{Multifractal analysis of weighted ergodic averages}

\date{}

\author{Aihua Fan}

\address[Aihua FAN]{LAMFA, UMR 7352 CNRS, University of Picardie,
33 rue Saint Leu, 80039 Amiens, France}
\email{ai-hua.fan@u-picardie.fr}

\begin{abstract}
	We propose to study the multifractal behavior of\\
	 weighted ergodic averages. Our study in this paper is concentrated on the symbolic dynamics. We introduce a thermodynamical formalism  which leads to a multifractal spectrum. It is proved that this thermodynamical formalism applies to different kinds of dynamically defined  weights, including stationary ergodic random weights, uniquely ergodic weights etc. 
	But the validity of the thermodynamical formalism for very irregular weights, like M\"{o}bius function,  is an unsolved problem. The paper ends with some other unsolved problems.
\end{abstract}

\maketitle

\section{Introduction}

For a given topological dynamical system 
$(X, T)$,  a continuous function $f\in C(X)$ and a sequence
of weights $w=(w_n) \subset \mathbb{C}$ such that $\sum_{n=0}^\infty|w_n|=\infty$,  we would like to study the asymptotic behavior of the weighted Birkhoff sums
$$
S_N^{(w)} f(x) = \sum_{n=0}^{N-1} w_n f(T^n x). 
$$ 
One of associated problems is the multifractal analysis of $S_N^{(w)}f(x)$. This  is a difficult problem even for simple dynamical systems when the sequence of weights is irregular like the M\"{o}bius function $\mu: \mathbb{N}\to \{-1, 0, 1\}$. Recall that $\mu(1) =1$; $\mu(n)=(-1)^k$ if $n=p_1\cdots p_k$ a product of $k$ distinct primes;  $\mu(n)=0$ otherwise. 
\medskip
 
 The usual Birkhoff sums (with constant weight $w_n=1$) were extensively studied in the literature for different dynamical systems (\cite{BD2004}, \cite{BSS2002}, \cite{BSS2002b}, \cite{Chung2010},
 \cite{FF}, \cite{FFW}, \cite{FLP}, \cite{FLW2002}, \cite{Hofbauer2007}, 
 \cite{IJ2015},
  \cite{JJOP2007},
 \cite{KS2007},
 \cite{Olivier1999},
 \cite{Olsen2003},   \cite{Olsen2006},  \cite{Olsen2008},  \cite{OlsenW2007}, 
 \cite{PW2001}, \cite{Reeve2011}, \cite{Reeve2012}, \cite{TV1999}, \cite{TV2003},
 \cite{Thompson2009}, \cite{Tian2017}).
 Some variants or generalizations of the usual Birkhoff sums were also well studied (\cite{BF2012}, \cite{B1996}, \cite{BCW2013},
 \cite{FJLR2015}, \cite{FLM2010}, \cite{FLWW2009},  \cite{Feng2003}, 
 \cite{Feng2009}, \cite{FH2010}, \cite{IJ2015b}, \cite{PS2007}, \cite{PW1999}).
 For surveys on the topic, see  \cite{Climenhaga2014}, \cite{Fan2014}, \cite{PW1997}. 
 \medskip
 
In this paper we consider the symbolic dynamics $(X, T)$ where  $X:=S^{\mathbb{N}}$,  $S$ being a finite set of $q$ elements ($q \ge 2$  being an integer),  and $T$ is the
shift transformation defined by $(Tx)_{n} =x_{n+1}$ for $x=(x_n) \in S^{\mathbb{N}}$. 
 Let us assume that $(w_n)$ is bounded so that
$w_n f(T^n x)$ is a bounded sequence of functions  in $C(X)$. 
So, 
more generally,  we consider
\begin{equation}\label{WES}
     S_N f(x) := \sum_{n=0}^{N-1} f_n(T^n x) =\sum_{n=0}^{N-1} f_n(x_n, x_{n+1}, \cdots)
\end{equation}
where $f=(f_n) \subset C(X)$ is 
a sequence of continuous functions such that $\|f_n\|_\infty=O(1)$, where $\|\cdot\|_\infty$ denotes the supremum norm in $C(X)$.
Define the lower and upper weighted Birkhoff averages by
$$
    \underline{A}(x) := \liminf_{N\to \infty}
    \frac{S_Nf(x)}{N};
    \quad 
    \overline{A}(x) := \limsup_{N\to \infty}
    \frac{S_Nf(x)}{N}.
$$
For $-\infty <a \le b <+\infty$, we define the level set
$$
    E([a, b]) :=\{x \in X: a \le \underline{A}(x) \le \overline{A}(x)\le b\}.
$$
If $[a, b]$ reduces to a singleton $\{a\}$, we write $E(a)$ instead of  
$E([a, a])$. 
The space $X$ is equipped with its natural distance
defined by $d(x, y) = q^{-n}$ where $n$ is the least $k$ such that 
$x_k\not=y_k$ and we can then
define the Hausdorff dimension $\dim E$ and the packing dimension ${\rm Dim}\ E$ of a set $E$ (see \cite{Falconer1990} or \cite{Mattila1995} for the definitions). 
We are concerned with the multifractal analysis of $S_N f$
defined by (\ref{WES}), in other words,  we would like to compute the dimensions of $E([a, b])$'s.

To this end, let us present the following thermodynamical formalism, adapted from that introduced in \cite{Fan1997}. Let $dx = \sigma(dx)$ denote the uniform
Bernoulli measure on $X:=S^\infty$, which is defined by
$$
     \sigma ([x_0, x_1,\cdots, x_{n-1}])
     = q^{-n}
$$
where $[x_0, x_1, \cdots, x_{n-1}]$ is the cylinder set consisting of all points $y$ such that $y_k=x_k$ for $0\le k <n$. This Bernoulli measure, which is $T$-invariant,  is  
our reference measure. For any real number $\lambda \in \mathbb{R}$, define
$$
P_{n} (x) :=\exp  \left(\sum_{k=0}^{n-1} f_k(x_k, x_{k+1},\cdots )\right); \quad   Z_n(\lambda)
: = \mathbb{E} P_{n}^\lambda(x) 
$$
where the expectation $\mathbb{E}$ is relative to the Bernoulli measure $\sigma$. 
In the sequel, 
we will make the following assumptions
\begin{equation*}\label{H1}
	{\rm (H1)} \quad
	\qquad \qquad \forall \lambda \in \mathbb{R}, \quad\phi(\lambda)
	: = \lim_{n \to \infty} \frac{1}{n} \log Z_n (\lambda)\ \ {\rm exists}. \qquad \qquad
\end{equation*} 
\begin{equation*}\label{H2}\
	{\rm (H2)} \ \  \ \sup_{N\ge 1} \ \ \sup_{x_0=y_0, \cdots, x_{N-1}=y_{N-1}}\ \  \sum_{k=0}^{N-1} |f_k(x_{k+1}, \cdots) - f_k(y_{k+1}, \cdots)| <\infty.
\end{equation*}

The function $\phi$ is convex then continuous. It is called the {\bf pressure function}, associated to $(f_n)$. Recall that 
the {\bf subderivative} of $\phi$ at $\lambda$, denoted $\partial \phi(\lambda)$,
is the set of real numbers $d$'s such that
$$
\forall \eta \in \mathbb{R}, \ \ \ \phi (\lambda +\eta) - \phi(\lambda) \ge d \eta.
$$
Notice that $\partial \phi(\lambda)$ is a closed interval. 
The {\bf conjugate} of a convex function $\phi$ on $\mathbb{R}$ is defined by
$$
\phi^*(\beta) = \sup_{\alpha\in \mathbb{R}} (\beta \alpha -\phi(\alpha)) \qquad (\forall \beta \in \mathbb{R}),
$$
which is a convex function too. For all these notions and facts on convex functions we can refer to \cite{Ellis}. Let
$$
\psi(\lambda):= \phi(\lambda) + \log q.
$$
It would be better to call $\psi$ the pressure function, because,  when
$f_n(x) = f(T^n)$ for all $n\ge 1$ ($f$ being fixed), $\psi(\lambda)$ is exactly the pressure of $\lambda f$ in the usual sense (see \cite{Bowen1975})
defined by 
$$
     \lim_{N\to\infty}\frac{1}{N} \log \sum_{a_0, a_1, \cdots, a_{N-1} } \ \ \sup_{x \in [a_0, a_1, \cdots, a_{N-1}]} e^{\lambda \sum_{k=0}^{N-1} f(T^k x)}.
$$

One of our main results is the following theorem.

\begin{theorem}\label{thm:main1} Suppose that the assumptions (H1) and (H2) are satisfied. 
	For $\lambda >0$, we have 
	$$
-	\frac{\psi^*(\max \partial \psi(\lambda))}{\log q} \le     \dim E(\partial \psi(\lambda)) \le {\rm Dim } E(\partial \psi(\lambda)) \le - \frac{\psi^*(\min \partial \psi(\lambda))}{\log q}.
	$$  
	For $\lambda<0$, we have similar estimates but we have to exchange the roles of  $\min \partial \psi(\lambda)$  and $\max \partial \psi(\lambda)$.
\end{theorem}

The ideal case is when $\psi$ is differentiable. Then 
$\partial \psi(\lambda)$ reduces to a singleton and we get 
equalities instead of inequalities in Theorem \ref{thm:main1}.
In other words, for $\alpha =\psi'(\lambda)$ we have
\begin{equation}\label{eq:spectrum}
    \dim E(\alpha) = {\rm Dim } E(\alpha) = - \frac{\psi^*(\alpha)}{\log q} = \frac{\psi(\lambda)-\lambda \alpha}{\log q}.
\end{equation}
Thus, a natural problem is to prove the differentiability
of $\psi$ in concrete cases. We will prove this in  some  cases. 
\medskip

Let us apply Theorem \ref{thm:main1}  to   
$f_n(x)= w_n f(T^n x)$ where  the weights $(w_n)$ are dynamically defined. 
We say that $f$ is  of {\bf  bounded variation}, if  $\sum_{n=1}^\infty {\rm var}_n(f)<\infty$ with
$$
    {\rm var}_n(f) := \sup_{x_0=y_0, \cdots, x_{n-1}=y_{n-1}}|f(x) -f(y)|.
$$ 

First we apply Theorem \ref{thm:main1}   to the case of ergodic random stationary
weights. Especially in the case that $f$ depends only on a finite number of coordinates, the pressure function be will proved to be analytic.


\begin{theorem}\label{thm:random} Consider the case 
	$f_n(x) = \omega_n f(T^nx)$. Suppose that $(\omega_n)$
	is an ergodic sequence of real random variables with 
	$\|\omega_n\|_\infty=O(1)$ and that $f$ 
	is of bounded variation. Then \\
	\indent {\rm (a)} almost surely
	the assumptions (H1) and (H2) are satisfied and 
	the function $\phi$ is independent of $\omega$;\\ 
	\indent {\rm (b)} if, furthermore, $f$ depends only on a finite number of coordinates, then  $\phi$ is an analytic function of $\lambda \in \mathbb{R}$. 
\end{theorem}

As we will see, the first assertion of Theorem \ref{thm:random} follows from Kingman's subadditive ergodic theorem and the second assertion  follows from Ruelle's theorem \cite{Ruelle1979}. If $f$ depends only on a finite number of coordinates and if $(\omega_n)$ is a sequence of independent and identically distributed random variable taking a finite number of values, Pollicott's method in \cite{Pollicott2010} can allow us to  numerically compute $\psi$
and then to numerically find the multifractal spectrum presented by the formula (\ref{eq:spectrum}). 
\medskip

If the weight $(w_n)$ is realized by a uniquely ergodic dynamical system, it is natural to ask if the pressure exists for every 
such realization. The answer is confirmative under the extra
condition that $f$ depends only only a finite number of coordinates. The problem is actually converted to the existence of  maximum Liapounov exponent of matrix-valued cocycles.

\begin{theorem}\label{thm:UE} Consider the case 
	$f_n(x) = \phi(\Theta^n \omega) f(T^nx)$, where $\Theta: \Omega \to \Omega$ is a uniquely ergodic dynamical system and $\phi \in C(\Omega)$ is a continuous function and 
	$f\in C(X)$ 
	is a function depending only on a finite number of coordinates. 
	Then for every $\omega\in \Omega$,  the pressure function $\phi$
	is well defined and  independent of $\omega$, and is an analytic function of $\lambda$.
\end{theorem}

This follows essentially from a result due to Furman \cite{Furman1997}
and from Theorem \ref{thm:random}.
\medskip

Now assume that $(w_n)$ is a sequence taking a finite number of real values, say $A \subset \mathbb{R}$. The shift $\Theta: A^\infty \to A^\infty$ acts on the  closed orbit
$\overline{\{\Theta^k w\}}$. We will assume that the subsystem
$(\overline{\{\Theta^k w\}}, \Theta)$ is minimal and uniquely ergodic (then we will simply say that $w$ is minimal and uniquely ergodic).  Then the condition imposed in Theorem \ref{thm:UE} that $f$ depends on the first coordinates can be dropped for the function $\phi$ to be well defined. 
Let us point out that all primitive substitutive sequences are minimal and uniquely ergodic \cite{Michel1974}.

\begin{theorem}\label{thm:main2} Consider 
	$f_n(x) = w_n f(T^nx)$. Suppose that $(w_n)\in A^\infty$
	is minimal and uniquely ergodic  and that $f$ 
	is of bounded variation. Then 
	the assumptions (H1) and (H2) are satisfied. 
\end{theorem}

The proof of Theorem \ref{thm:main2} is based on the notion of return word (see \cite{Durand1998}). 
\medskip

\medskip

Now let us look at some particular cases for which we can find an explicit formula for  the function $\phi$ and then an explicit formula for the spectrum given by  (\ref{eq:spectrum}).
For simplicity, just consider the case $S=\{-1, 1\}$. 
A typical example of $f_n(x_n, x_{n+1}, \cdots )$ is of the form
\begin{equation}\label{Example}
      w_n (a+b x_n + c x_{n+1} + d x_n x_{n+1} + e x_n x_{2n} + f x_n x_{2n} x_{3n})
\end{equation} 
where $a,b,c,d, e,f$ are fixed constants. Special cases include
\begin{eqnarray}
	  f_n(x_n, x_{n+1}, \cdots) &=& w_n x_nx_{n+1}, \label{Case 1}\\
	  f_n(x_n, x_{n+1}, \cdots) &=& w_n x_nx_{2n},\label{Case 2} \\
	  f_n(x_n, x_{n+1}, \cdots) &=& w_n x_nx_{2n} x_{3n}. \label{Case 3} 
	\end{eqnarray}

When $w_n=1$ for all $n$, the first  case defined by (\ref{Case 1})   is  classical and well studied (for example, see \cite{FF, FFW}),  and the  cases defined by (\ref{Case 2}) and (\ref{Case 3}) are
 studied in \cite{FLM} using Riesz product measures. 
 Notice that we can not apply Theorem \ref{thm:main1} to the last two cases because the assumption (H2) is not satisfied.

In the following we discuss the case $f_n(x) = w_n x_n x_{n+1}$ with
some more or less regular weights $(w_n)$. By the way,  we will also discuss some   
generalizations of $w_n x_n x_{n+1}$. An explicit formula for $\phi$
will be obtained.
As a corollary of Theorem \ref{thm:main1}, we can then prove the following result. But we  can and we will provide a direct proof of the following theorem too. 

\begin{theorem}\label{thm1} Let $S=\{-1, 1\}$.
	Assume that $(w_n)$ takes a finite number of values $v_0, v_1, \cdots, v_m$ and that $(f_n)$ are of the form
	$$
	f_n(x) = x_n g_n(x_{n+1}, x_{n+2}, \cdots).
	$$
	 Suppose further that \\
	\indent {\rm (C1)}  all $g_n$ take   
	values in $\{-1, 1\}$ and there is an integer $L\ge 1$ such that
	$g_n(x_n, \cdots)$ depends only on $x_n, x_{n+1}, \cdots, x_{n+L}$; \\ 
	\indent {\rm (C2)} 
	the following frequencies exist
	\begin{equation}\label{eqn:freq}
	p_j := \lim_{N\to\infty} \frac{\#\{1\le n \le N: w_n=v_j\}}{N}
	\quad (0\le j \le m).
	\end{equation}
	Then for $\alpha \in (-\sum p_j |v_j|,  \sum p_j |v_j|)$ we have
	$$
	\dim E(\alpha) =  
	\frac{1}{\log 2}\sum_{j=0}^m p_j \left(\log (e^{\lambda_\alpha v_j} + e^{-\lambda_\alpha v_j})
	-  \lambda_\alpha v_j \frac{e^{\lambda_\alpha v_j}- e^{-\lambda_\alpha v_j}}{e^{\lambda_\alpha v_j}+e^{-\lambda_\alpha v_j}}\right)
	$$
	where $\lambda_\alpha$ is the unique solution of the equation
	$$
	\sum_{j=0}^m p_j v_j \frac{e^{\lambda_\alpha v_j}- e^{-\lambda_\alpha v_j}}{e^{\lambda_\alpha v_j}+e^{-\lambda_\alpha v_j}} = \alpha.
	$$
	\end{theorem}

Notice that the result doesn't depend on the form of $g_n$, but only on the weights $(w_n)$. The key point for this independence is that $g_n$ only takes $-1,1$
as its values.

As a corollary of Theorem \ref{thm1}, we have the following result for
$$
  F(\alpha) =\left\{x \in \{-1,1\}^\mathbb{N}: \lim_{N\to\infty} \frac{1}{N}\sum_{n=1}^N \mu(n) x_nx_{n+1} =\alpha \right\}
$$
where $\mu$ is the M\"{o}bius function.

\begin{theorem}\label{thm2} For any $\alpha \in (-6/\pi^2, 6/\pi^2)$, we have
	$$
	\dim  F(\alpha) = 1- \frac{6}{\pi^2} + 
	\frac{6}{\pi^2 \log 2} H\left( \frac{1}{2} + \frac{ \pi^2}{12} \alpha \right)
	$$
	where $H(x) = -x \log x - (1-x)\log (1-x)$.

\end{theorem}


{\em Acknowledgement.}  The author would like to thank D. Ruelle for ensuring  the correct use of  his result on the analyticity of Liapounov exponent, and Lingmin Liao for his careful reading the first version of the paper.  Thanks go to Meng Wu for his idea used in the proof of Theorem \ref{thm:main2}. 
The main results in the paper were presented in a mini-course  during the 2019 Fall Program of Low-Dimensional Dynamics held in Shanghai Center for Mathematical Sciences and the author is grateful to Weixiao Shen for providing this opportunity. The main results as well as unsolved problems were also informally presented in the conference
Fifty Years of Thermodynamic Formalism  in Physics, Dynamics and  Probability
(Leiden, August 2018) and the author is grateful to Evgeny Verbitskiy, Frank Redig and Antony Quas for stimulating discussions.


\section{Thermodynamical formalism: Proof of Theorem \ref{thm:main1}}
\label{sect:2}
We present here a thermodynamic formalism  proposed in 
\cite{Fan1997},
adapted to our setting in the present paper. 
This formalism would work in other cases. As we will point out in the last section, there will be works to do with the  limit defining the pressure  and with the differentiability of the pressure.

\subsection{Fundamental inequalities}

For $m\le n$, denote
   $$
      P_{m, n} (x) := \frac{P_n(x)}{P_m(x)}; \quad Z_{m, n}(\lambda)
      := \mathbb{E} P_{m, n}^\lambda(x).
   $$
   
The following inequalities are fundamental. In \cite{Fan1997} (p. 1318),
the inequalities are stated in a more general case and proved in a different way.

\begin{lemma}[Fundamental inequalities]  \label{lem:ZZZ}
For any real number $\lambda \in \mathbb{R}$, there exists a positive constant $C(\lambda)>0$
such that for all integers   $0\le l\le m\le n$ we have 
	\begin{equation}\label{ieq:fund}
	    \frac{1}{C(\lambda)} \le \frac{Z_{l, n}(\lambda)}{Z_{l, m}(\lambda) Z_{m, n}(\lambda)} \le C(\lambda).
	\end{equation}
	The constant $C(\lambda)$ grows at most exponentially as function of $\lambda$,
	i.e. $C(\lambda) = e^{O(|\lambda|)}$. 
	\end{lemma} 

\begin{proof} For all integers $0\le m\le n$, let 
	$$S_{m,n}f(x)=\sum_{k=m}^{n-1} f_k(T^k x) = \sum_{k=m}^{n-1} f_k( x_k,x_{k+1},\cdots).
	$$ 
	Then $P_{m, n} = e^{S_{m,n}f}$ (by convention $S_{n,n}f=0$ so that
	$P_{n,n}=1$). First we notice that for any $(a_0, a_1,\cdots, a_{m-1})\in S^m$ we have
	\begin{equation}\label{eq:local-int}
	       Z_{m, n}(\lambda) = q^m \int_{[a_0, a_1, \cdots, a_{m-1}]}
	        e^{\lambda S_{m,n} f(x)} d\sigma(x).
	\end{equation}
	Indeed, the map $T^m: [a_0, a_1,\cdots, a_{m-1}] \to X$
	is bijective and it maps the probability measure 
	$q^m \sigma|_{[a_0, a_1, \cdots, a_{m-1}]}$ to the probability measure $\sigma$. So, by a change of variables,  the member at the right hand side of 
	(\ref{eq:local-int}) is equal to
	$$
	q^m \int_{[a_0, a_1, \cdots, a_{m-1}]}
	e^{\lambda \sum_{k=m}^{n-1} f_k(T^k x)} d\sigma(x)
	= \int_X
	e^{\lambda \sum_{k=m}^{n-1} f_k(T^{k-m} x)} d\sigma(x).
	$$
	Then, by the $T$-invariance of $\sigma$, we get
	$$
	   \int_X
	   e^{\lambda \sum_{k=m}^{n-1} f_k(T^{k-m} x)} d\sigma(x)
	   =\int_X
	   e^{\lambda \sum_{k=m}^{n-1} f_k(T^{k} x)} d\sigma(x)=Z_{m,n}(\lambda).
	   $$ 
	   Thus, (\ref{eq:local-int}) is proved. Now write 
     $$
         Z_{l,n}(\lambda)=\sum_{a_0, a_1, \cdots, a_{m-1}}
                \int_{[a_0, a_1, \cdots, a_{m-1}]}
                    e^{\lambda S_{l,m} f(x) + \lambda S_{m,n}f(x)} d\sigma(x).
     $$
     By the distortion hypothesis (H2), we have  
     $$
     Z_{l,n}(\lambda) \approx \sum_{a_0, a_1, \cdots, a_{m-1}}
         q^{-m} e^{\lambda S_{l,m} f(a_0, a_1, \cdots, a_{m-1}, *)}
         	\cdot q^m
     \int_{[a_0, a_1, \cdots, a_{m-1}]}
     e^{\lambda S_{m,n}f(x)} d\sigma(x).
     $$
     where $*$ represents any fixed sequence, and the constant involved
     in "$\approx$" is $e^{O(|\lambda|)}$. By (\ref{eq:local-int}),
     the above expression reads as
     $$
          Z_{l,n}(\lambda) \approx \sum_{a_0, a_1, \cdots, a_{m-1}}
          q^{-m} e^{\lambda S_{l,m} f(a_0, a_1, \cdots, a_{m-1}, *)}
          \cdot Z_{m, n}(\lambda).
     $$
     Using 
      once more the hypothesis (H2), we get that the last sum is equal to
     $Z_{l, m}(\lambda)$ up to a multiplicative constant 
     $e^{O(|\lambda|)}$. 
	
	\end{proof}	

We emphasize that the equality (\ref{eq:local-int}) is a key point.

\subsection{Construction of Gibbs measure}

Let
$$
     d \mu_{n, \lambda} : = \frac{P_n^\lambda(x)}{Z_n(\lambda)} dx. 
$$
It is a probability measure on $X$.

\begin{lemma}[Gibbs property]  
	 \label{lem:Gibbs} All weak limits of  
	the sequence of probability measures $(\mu_{n, \lambda})$  are equivalent.   
	For any such a limit, denoted by $\mu_{\lambda}$,  we have
	$$
	  \mu_{\lambda} ([x_0, x_1, \cdots, x_{n-1}]) \approx \frac{P_n^\lambda(x)}{q^n Z_n(\lambda)}.
	$$ 
	The constant involved in "$\approx$" depends on $\lambda$ but is independent of $n$ and $x$, and is of the size $e^{O(|\lambda|)}$.
\end{lemma} 

\begin{proof} The proof is already in (\cite{Fan1997}, p.1319). It is simpler in the present case. For completeness, we include it here. 
	Let $C_n(x)$ be the cylinder $[x_0,x_1, \cdots, x_{n-1}]$. Assume that $\mu_{\lambda}$ is the weak limit of 
	$(\mu_{N_j, \lambda})$ for some sequence of integers  $(N_j)$. We have 
	\begin{eqnarray*}
		\mu_\lambda(C_n(x)) 
		& = &   \lim_{j\to\infty}\frac{1}{Z_{N_j}(\lambda)}\int_{C_n(x)}     
		P_{N_j}^\lambda(y)dy\\
		& \le& C'    \limsup_{j\to\infty}\frac{1}{Z_n(\lambda)Z_{n, N_j}(\lambda)}\int_{C_n(x)}     
		P_n^\lambda (y) P_{n, N_j}^\lambda(y)dy\\
			& \le & C''    
		 \limsup_{j\to\infty}\frac{P_n^\lambda(x)}{q^n Z_n(\lambda)}\cdot \frac{q^n}{Z_{n, N_j}(\lambda)}\int_{C_n(x)}     
		P_{n, N_j}^\lambda(y)dy\\
		& = &   C''
	\frac{P_n^\lambda(x)}{q^n Z_n(\lambda)}.	
		\end{eqnarray*}  
	The first inequality above is a consequence of the fundamental
	inequalities (\ref{ieq:fund}); the second inequality is a consequence of the distortion hypothesis (H2) and the last equality is because of (\ref{eq:local-int}).  
	
	The inverse inequality can be proved in the same way, because we have both side estimates in our fundamental inequalities.  
	\end{proof}

The measure $\mu_{\lambda}$ will be called {\bf Gibbs measure} associated to $(f_k)$ and
$\lambda$.  Fix $n\ge 1$. Define 
$$
    \forall k\ge 0, \ \ g_k(x) = f_{n+k}(x)
$$
The Gibbs measure associated to $(g_k)$ and $\lambda$ will be denoted by $\mu_\lambda^{(n)}$. This measure depends on the tail from $n$ on of $(f_k)$.

 Notice that $\mu_\lambda=\mu_\lambda^{(0)}$.
The family $(\mu_\lambda^{(n)})$ has the following quasi-Bernoulli
property, which is a direct consequence of the above Lemma \ref{lem:ZZZ} and Lemma \ref{lem:Gibbs}.

\begin{lemma}[Quasi-Bernoulli property] \label{lem:quasi-Bernoulli}
	 For all integers $n$ and $m$ and for all sequences $I\in S^n$ and $J\in S^m$, we have 
	 $$
	      \mu_{\lambda}([IJ]) \approx  \mu_{\lambda}([I])  \mu_{\lambda}^{(n)}([J])
	 $$
	 where the  constants involved in "$\approx$" are independent of $n, m$ and $I,J$.
	\end{lemma}

\subsection{Large deviation}

We are going to present a law of large numbers with respect to our Gibbs measures. It is a consequence of a well known result on large deviation.
The large deviation was used in multifractal analysis in early works (see \cite{BMP1992}, for example).
The following result on convex functions and their conjugates will be useful.

\begin{proposition}[\cite{Ellis}, p.221]\label{prop:phi}
	For any convex function defined on $\mathbb{R}$, we have
	\begin{itemize}
		\item[(i)]   $\alpha \beta \le \phi(\beta) + \phi^*(\alpha)$,  \ ($\forall \alpha, \beta \in \mathbb{R}$).\\
		\item[(ii)]   $\alpha \beta = \phi(\beta) + \phi^*(\alpha) \Longleftrightarrow \alpha \in \partial \phi(\beta)$.\\
		\item[(iii)]   $ \alpha \in \partial \phi(\beta) \Longleftrightarrow \beta \in \partial \phi^*(\alpha)$.\\
		\item[(iv)] $\phi^{**}(\beta) =\phi(\beta)$, ($\forall \beta \in \mathbb{R}$).
	\end{itemize}
\end{proposition}

Let $(W_n)$ be a sequence of random variables on a probability space $(\Omega, \mathcal{A}, \nu)$ and $(a_n)$ be a sequence of positive real numbers tending to the infinity. Suppose that the following limit exits
$$
   c(\beta):=c_W(\beta):= \lim_{n\to \infty} \frac{1}{a_n} \log \mathbb{E} e^{\beta W_n}, \quad \forall \beta \in \mathbb{R}.
$$
We call $c(\beta)$ the {\bf free energy function} of $(W_n)$ with respect to $\nu$ and weighted by $(a_n)$. By the upper large deviation bound theorem (\cite{Ellis}, p. 230), for any non empty compact set $K\subset \mathbb{R}$ we have
$$
    \limsup_{n\to \infty} \frac{1}{a_n} \log \sigma\big\{ a_n^{-1} W_n \in K\big\}
    \le - \inf_{\alpha \in K} c^*(\alpha).
$$ 
Notice that $c(0)=0$. By Proposition \ref{prop:phi} (i), we have $c^*(\alpha)\ge 0$ for all $\alpha \in \mathbb{R}$. By Proposition \ref{prop:phi} (ii), we have $c^*(\alpha) =0$ iff $\alpha \in \partial c(0)$. 
Let 
$$\Delta^- := \min \partial c(0), \quad\Delta^+ := \max \partial c(0).
$$ 
For any   compact set $K$ disjoint  from $[\Delta^-, \Delta^+]$, 
we have $\eta:=\inf_K c^*(\alpha) >0$. 
By the upper large deviation bound theorem, for large $n$ we have
$$
    \nu\big\{  a_n^{-1} W_n \in K  \big\} \le e^{- \eta a_n/2}.
$$
Suppose that $\sum e^{-\epsilon a_n }<\infty$ for all $\epsilon >0$
(it is the case when $a_n=n$). By the Borel-Cantelli lemma, we get 
\begin{equation}\label{LLN1}
   \nu\!-\!a.e \quad    \min \partial c(0) \le \liminf_{n\to \infty} \frac{W_n}{a_n}
      \le \limsup_{n\to \infty} \frac{W_n}{a_n} \le \max \partial c(0).
\end{equation}

Now fix a Gibbs measure $\mu_{\lambda}$. We consider the free energy
of $S_nf(x)$
relative to $(\mu_{\lambda}, \{n\})$ defined as follows
$$
     c_\lambda (\beta) := \lim_{n\to\infty} \frac{1}{n} \log \mathbb{E}_{\mu_{\lambda}} P_n^\beta(x).
$$

\begin{lemma}[\cite{Fan1997}, p.1322] Suppose the limits defining $\phi(\lambda)$ exist. Then the limit defining $c_\lambda(\beta)$
	exists and we have  
	$$
	c_\lambda(\beta) = \phi(\lambda+\beta) - \phi(\lambda).
	$$ 
\end{lemma} 

	It follows  that $c_\lambda(0) =0$ so that $c_\lambda^*(\alpha)\ge 0$ for all $\alpha$. Also
	$
	    c_\lambda^*(\alpha) =0 
	$ iff $\alpha \in \partial \phi(\lambda)$. 
	Then we can apply (\ref{LLN1}) to get the following law of large numbers.
	
	\begin{lemma}[Law of large numbers] \label{LLN2} For $\mu_\lambda$-almost all $x$,  we have
	$$
	          \min \partial \phi(\lambda) \le \underline{A}(x)
	          \le  \overline{A}(x)
	          \le \max \partial \phi(\lambda). 
	$$   \end{lemma}

It will be more practical to work with
$$
    \psi(\lambda):= \phi(\lambda) + \log q.
$$
The above inequalities in Lemma \ref{LLN2} also hold with $\psi$ replacing $\phi$.

\subsection{Dimensions of Gibbs measures}
The local lower and upper dimensions of a measure $\mu$ are respectively defined by
\begin{eqnarray*}
    \underline{D}(\mu,x) & = & \liminf_{n\to \infty}
    \frac{\log \mu([x_0, x_1, \cdots, x_{n-1}])}{\log q^{-n}}, \\
    \overline{D}(\mu,x) &=& \limsup_{n\to \infty}
    \frac{\log \mu([x_0, x_1, \cdots, x_{n-1}])}{\log q^{-n}}
\end{eqnarray*}
The lower and upper Hausdorff dimensions of $\mu$ are respectively defined by
\begin{eqnarray*}
     \dim_*\mu &=& \inf\{\dim E: \mu(E)>0\}, \\
     \dim^*\mu &=& \inf\{\dim E: \mu(E^c)=0\}.
\end{eqnarray*}
The lower packing dimension ${\rm Dim}_*\mu$ and the upper packing dimension ${\rm Dim}^* \mu$   are similarly defined by using 
the packing dimension ${\rm Dim} E$ instead of the Hausdorff dimension
$\dim E$. 

A systematic study of the Hausdorff dimensions $\dim_*\mu$ and $\dim^*\mu$ was carried out in \cite{Fan1989, Fan1994} when $X$ is a homogeneous space. Later, the packing dimensions 
${\rm Dim}_*\mu$ and ${\rm Dim}^*\mu$ were studied independently  by Tamashiro \cite{Tamashiro1995} and Heurteaux \cite{Heurteaux1998}. 
Let us just state the following result. 
\medskip

\begin{proposition}[\cite{Fan1989, Fan1994,Heurteaux1998,Tamashiro1995}]\label{prop:fan} For the Hausdorff dimensions we have
	$$
	\dim_*\mu = {\rm ess inf}_\mu \underline{D}(\mu, x),
	\quad 
	\dim^*\mu = {\rm ess  sup}_\mu \underline{D}(\mu, x).
	$$
Similar formulas hold for the packing dimensions ${\rm Dim}_* \mu$
and ${\rm Dim}^*\mu$ if we replace $\underline{D}(\mu, x)$
by $\overline{D}(\mu, x)$.
	\end{proposition}

From the Gibbs property (Lemma \ref{lem:Gibbs}), we get immediately
the following relation between the local dimensions of a Gibbs measure and the averages $\underline{A}(x)$ and $\overline{A}(x)$.

\begin{lemma}[Local dimensions of Gibbs measures]\label{lem:localdim}
	  For all $x\in X$ we have
	  $$
	       \underline{D}(\mu_{\lambda}, x) = \frac{\psi(\lambda ) - \lambda \overline{A}(x)}{\log q} 
	       \ {\rm if }\ \lambda >0; 
	       \quad  \underline{D}(\mu_{\lambda}, x) = \frac{\psi(\lambda ) - \lambda \underline{A}(x)}{\log q} 
	       \ {\rm if }\ \lambda <0. 
	       \quad .
	  $$
	  Similar equalities hold when $\underline{D}(\mu_{\lambda}, x)$
	  is replaced by $\overline{D}(\mu_{\lambda}, x)$ and 
	  $\overline{A}(x)$ by $\underline{A}(x)$.
	\end{lemma}

The measure $\mu_0$ is the symmetric Bernoulli measure and its dimension is equal to $1$. The dimension of the Gibbs measures are estimated as follows.

\begin{lemma}[Dimensions of Gibbs measures] \label{lem:dim} \ \\
	\indent {\rm (1)}  If $\lambda >0$, we have
	$$
	   - \frac{\psi^*(\max \partial \psi(\lambda))}{\log q} \le \dim_*\mu_{\lambda} \le \dim^* \mu_{\lambda} \le - \frac{\psi^*(\min \partial \psi(\lambda))}{\log q}. 
	$$
		\indent {\rm (2)}  If $\lambda <0$, we have similar estimates but we have to exchange the positions of 
		 $\max \partial  \psi(\lambda)$ and  $ \min \partial  \psi(\lambda)$.\\
		 	\indent {\rm (3)} We have exactly the same estimates for the packing dimensions ${\rm Dim}_* \mu_{\lambda}$
		 	and ${\rm Dim}^* \mu_{\lambda}$. 	
\end{lemma}
\begin{proof} (1) From Lemma \ref{lem:localdim}, Lemma \ref{LLN2}
	and Proposition \ref{prop:fan}, we  get 
	$$
	\frac{\psi(\lambda)-\lambda \max \partial  \psi(\lambda)}{\log q} \le \dim_*\mu_{\lambda} \le \dim^* \mu_{\lambda} \le \frac{\psi(\lambda)-\lambda \min \partial \psi(\lambda)}{\log q}. 
	$$
	But, by Proposition \ref{prop:phi} (ii), we have
	\begin{eqnarray*}
	    \psi(\lambda)-\lambda \min \partial \psi(\lambda)
	    &=&
	    -\psi^*(\min \partial \psi(\lambda)),
	    \\
	     \psi(\lambda)-\lambda \max \partial \psi(\lambda)
	     &=&
	    -\psi^*(\max \partial \psi(\lambda)). 
	\end{eqnarray*}

(2)  It is the same argument, but we have to exchange the roles of 
 $\max \partial \psi(\lambda)$ and  $\min \partial \psi(\lambda)$ in the above inequalities.
 Notice that
 $\overline{A}(x)$ and $\underline{A}(x)$ have the same bounds 
 in Lemma \ref{LLN2}.
 
 (3) It is the exactly the same argument as in (1) and (2).
	
	\end{proof}

\subsection{Proof of Theorem \ref{thm:main1}}
Now we are ready to prove Theorem \ref{thm:main1}.
Let $\partial \phi(\lambda)=[\alpha_-, \alpha^+]$.

 Assume $\lambda >0$. The fact $\underline{A}(x) \ge \alpha_-$
 for $x \in E([\alpha_-, \alpha^+])$  implies that the set
$E([\alpha_-, \alpha^+])$ is contained in
\begin{equation}\label{eq:lp}
   \bigcap_{\epsilon >0} \bigcup_{N \ge 1} \bigcap_{n >N}
   \left\{  n(\alpha_- -\epsilon)  {\color{red}}\le \sum_{k=0}^{n-1} f_k(x)\right\}. 
\end{equation}
By the $\sigma$-stability of the packing dimension, to upper bound the packing dimension of $E([\alpha_-, \alpha^+])$ it suffices to estimate the Minkowski dimension of the last set of intersection. i.e. $\cap_{n\ge N}\{ n(\alpha_- -\epsilon)\le S_nf(x)\}$. 
Consider the family $\mathcal{C}_n$ of all the $n$-cylinders  intersecting that set of intersection.  For any $d>0$ we have
\begin{eqnarray*}
	    \sum_{[x_0, x_1, \cdots, x_{n-1}]\in \mathcal{C}_n}
	        q^{- n d} 
	        & \le & 
	         \sum_{[x_0, x_1, \cdots, x_{n-1}]\in \mathcal{C}_n}
	        q^{- n d}\cdot  \frac{e^{\lambda \sum_{k=0}^{n-1} f_k(x)}}{e^{\lambda n(\alpha_- -\epsilon)}}\\
	        & \le &
	        \sum_{[x_0, x_1, \cdots, x_{n-1}]}
	        q^{- n d}\cdot  \frac{e^{\lambda \sum_{k=0}^{n-1} f_k(x)}}{q^n Z_n(\lambda)} \cdot \frac{q^n Z_n(\lambda)}{e^{\lambda n(\alpha_- -\epsilon)}}.
	    \end{eqnarray*}
By the Gibbs property of $\mu_{\lambda}$ (Lemma \ref{lem:Gibbs}), we have
\begin{eqnarray*}
	\sum_{[x_0, x_1, \cdots, x_{n-1}]\in \mathcal{C}_n}
	q^{- n d} 
	& \le &     
	        C q^{- n d}\cdot \frac{q^n Z_n(\lambda)}{e^{\lambda n(\alpha_- -\epsilon)}} \sum \mu_{\lambda} ([x_0, \cdots, x_{n-1}])\\
	        &\le& C  q^{- n d} q^n e^{n (\phi(\lambda) +\epsilon) - n \lambda (\alpha_- -\epsilon)}\\
	         &= & C  q^{- n [d  - \frac{1}{\log q}(\psi(\lambda) - \lambda \alpha_- + \epsilon +\lambda \epsilon)] }	\le C        
	\end{eqnarray*}
if $d> \frac{1}{\log q}(\psi(\lambda) - \lambda \alpha_- + \epsilon +\lambda \epsilon)$. It follows that the Minkowski dimension of the set in question is smaller than 
$\frac{1}{\log q}(\phi(\lambda) - \lambda \alpha_- + \epsilon +\lambda \epsilon)$. Let $\epsilon \to 0$, we get
\begin{equation}\label{ieq:+upper}
    {\rm Dim} E(\partial \psi(\lambda)) \le  \frac{\psi(\lambda) - \alpha_- \lambda}{\log q} = - \frac{\psi^*(\min \partial \psi(\lambda))}{\log q}.
\end{equation}
If $\lambda <0$, we can similarly prove 
\begin{equation}\label{ieq:-upper}
{\rm Dim}  E(\partial \psi(\lambda)) \le  \frac{\psi(\lambda) - \alpha_- \lambda}{\log q} = - \frac{\psi^*(\max \partial \psi(\lambda))}{\log q},
\end{equation}
but we must start with the fact that 
$E_{[\alpha_-, \alpha^+]}$ is contained in
\begin{equation}\label{eq:lm}
\bigcap_{\epsilon >0} \bigcup_{N \ge 1} \bigcap_{n >N}
\left\{  n(\alpha_+ + \epsilon)  {\color{red}}\ge \sum_{k=0}^{n-1} f_k(x)\right\}. 
\end{equation}
Notice that we have opposite inequalities in (\ref{eq:lp})
and  (\ref{eq:lm}).

Prove now the lower bound.  By Lemma \ref{LLN2}, $E(\partial \psi(\lambda))$ is of full $\mu_{\lambda}$-measure.
In particular, $\overline{A}(x) \le \max \partial \psi(\lambda)$
for $\mu_{\lambda}$-almost every $x$.  If $\lambda >0$, by Lemma \ref{lem:Gibbs}, this implies
$$\mu_{ \lambda}\!-\!a.e. x \quad \underline{D}(\mu_{\lambda}, x) \ge \frac{\psi(\lambda)-\lambda \max \partial \psi(\lambda) }{\log q}.$$
Thus
\begin{equation}\label{ieq:+lower}
    \dim E(\partial \psi(\lambda)) \ge -\frac{\psi^*(\max \partial \psi(\lambda))}{\log q}.
\end{equation}
When $\lambda <0$, we use the fact $\underline{A}(x) \ge \min \partial \psi(\lambda)$ for $\mu_{\lambda}$-almost every $x$ to get
\begin{equation} \label{ieq:-lower}
\dim E(\partial \psi(\lambda)) \ge -\frac{\psi^*(\min \partial \psi(\lambda))}{\log q}.
\end{equation} 
The four inequalities (\ref{ieq:+upper}), (\ref{ieq:-upper}) (\ref{ieq:+lower}) and  (\ref{ieq:-lower}) are what we have to prove.


\subsection{$\tau$-function of the Gibbs measure $\mu_{ \lambda}$}

For the Gibbs measure $\mu_{\lambda}$, we define the function
$$
\tau_\lambda(\beta) = \lim_{n\to\infty} \frac{1}{\log q^{-n}}
\log \sum_{x_0, x_1,\cdots, x_{n-1}} \mu_\lambda([x_0, x_1, \cdots, x_{n-1}])^\beta.
$$
The function $\tau_{\lambda}$ and the function $\phi$
has a simple explicit relation. The differentiability of $\phi$
at $\lambda$ is equivalent to the differentiability of $\tau_{\lambda}$ at $1$.

\begin{lemma}[Relation between $\tau$ and $\phi$] \label{lem:tau_lambda}
	Under the assumptions (H1) and (H2), the limit defining 
	$\tau_\lambda(\beta)$ exists for all $\beta \in \mathbb{R}$ and we have
	\begin{equation}\label{eq:tau_phi}
	    - \tau_\lambda(\beta) = 1 + \frac{\phi(\beta \lambda)-\beta\phi(\lambda)}{\log q}.
	\end{equation}
	Consequently, $\phi$ is differentiable at $\lambda$ iff $\tau_\lambda$ is differentiable at $1$. In this case we have
	\begin{equation}\label{eq:tau_phi2}
	      \tau_\lambda'(1) = \frac{\phi(\lambda) - \lambda \phi'(\lambda)}{\log q}.
	\end{equation}
	\end{lemma}

\begin{proof} The equality (\ref{eq:tau_phi}) follows from
	the Gibbs property:
	$$
	   \sum_{x_0, x_1,\cdots, x_{n-1}} \mu_\lambda([x_0, x_1, \cdots, x_{n-1}])^\beta \asymp q^n \int_X \frac{e^{\beta \lambda}S_nf(x)}{Z_n(\lambda)^\beta} dx
	   \asymp q^n \frac{Z_n(\beta \lambda)}{Z_n(\lambda)^\beta}.
	$$
	\end{proof}

\section{Stationary weights: Proof of Theorems \ref{thm:random}
	}

\subsection{Proof of Theorem \ref{thm:random}}
Assume that $\omega =(\omega_n)$ takes values in an interval 
$I \subset \mathbb{R}$. Let $\Omega = I^\mathbb{N}$ and let
$\Theta$ denote the shift map on $\Omega$ so that 
$\Theta^n \omega = (\omega_{n+k})_{k\ge 0}$.
To express clearly the dependence on $\omega$, denote by $\mu_{ \lambda}^{(\omega)}$ the Gibbs measure corresponding to the weight $\omega\in \Omega$ and write
$$
\quad Z_n(\lambda, \omega) = \int_X e^{\lambda S_n^{(\omega)} f(x)} dx
$$
where 
$$
S_n^{(\omega)}f(x) =\sum_{k=0}^{n-1}\omega_k f(T^k x).
$$
The fundamental inequalities (\ref{ieq:fund}) read as
\begin{equation}\label{eq:FI}
\frac{1}{C(\lambda)} \le \frac{Z_{m+n}(\lambda, \omega)}{Z_{ n}(\lambda, \omega) Z_{n, n+m}(\lambda, \Theta^n\omega)} \le C(\lambda). 
\end{equation}

The condition (H2) is clearly satisfied.  By Kingman's ergodic theorem, almost surely the condition (H1)
is also satisfied, i.e. almost surely the following limit exists
\begin{equation}\label{eq:phi_omega}
\phi^{(\omega)}(\lambda) = \lim_{n\to\infty}\frac{1}{n} \log Z_n(\lambda, \omega)
\end{equation}
and $\phi^{(\omega)}(\lambda)$ is almost surely  equal to the function
$\widetilde{\phi}(\lambda)$ defined by
$$
\widetilde{\phi}(\lambda) = \lim_{n\to\infty}\frac{1}{n} \mathbb{E}\log Z_n(\lambda, \omega).
$$
We actually have
$$
\left| \widetilde{\phi}(\lambda) - 
\frac{1}{n} \mathbb{E}\log Z_n(\lambda, \omega)\right|
\le \frac{|\lambda|\log C(\lambda)}{n}.
$$

Now suppose that $f$ depends only on the first $r\ge 2$ coordinates ($\widetilde{\phi}(\lambda)$ is easy to compute when $r=1$), i.e. $f$ takes the form
$$
f(x) = f(x_0, x_1, \cdots, x_{r-1}).
$$
For fixed $\lambda\in \mathbb{R}$ and fixed $w\in \mathbb{R}$, let us define
a $q^{r-1}\times q^{r-1}$-matrix $$
A_w(\lambda) = (a_{u, v})_{(u,v)\in S^{r-1}\times S^{r-1}}$$
as follows: if the $(r-2)$-suffix of $u$ is equal to the 
$(r-2)$-prefix of $v$, i.e. $u=x_0x_1\cdots x_{r-2}$
and $v=x_1\cdots x_{r-1}$ for some $(x_0, x_1, \cdots, x_{r-1})\in S^r$, then 
$$
a_{u, v} = e^{\lambda w f(x_0, x_1, \cdots, x_{r-1})};
$$  
otherwise $a_{u, v}=0$. Since $S_n^{(\omega)}f(x)$ is locally constant on cylinders of length $n+r$, it is easy to see that
$$
Z_n(\lambda, \omega)
= \frac{1}{q^{n+r-1}} \|A_{\omega_0}(\lambda) A_{\omega_1}(\lambda) \cdots A_{\omega_{n-1}}(\lambda) \|
$$
where $\|A\|$ denotes the norm defined by the sum of all the entries of a non-negative matrix $A$. Observe that our
matrices $A_{\omega_n}(\lambda)$ are non-negative and that the 
product of any $r$ consecutive matrices are strictly positive.
So,
$$
Z_{nr}(\lambda, \omega) =
\frac{1}{q^{nr+r-1}} \|B_\omega(\lambda) B_{\Theta \omega}(\lambda)\cdots B_{\Theta^{n-1} \omega}(\lambda)\|
$$
where
$$
B_{\omega}(\lambda) = A_{\omega_0}(\lambda) A_{\omega_1}(\lambda) \cdots A_{\omega_{r-1}}(\lambda)
$$
which is a strictly positive matrix. So, 
$$
\widetilde{\phi}(\lambda)
=  \frac{1}{r} \lim_{n\to\infty} \frac{1}{n} \mathbb{E} \log \|B_\omega(\lambda) B_{\Theta \omega}(\lambda)\cdots B_{\Theta^{n-1} \omega}(\lambda)\|-\log q.
$$
The above limit
is  the largest characteristic exponent of the
random positive matrix $B_\omega(\lambda)$. Since $\lambda \mapsto B_\omega(\lambda)$ analytic and $B_\omega(\lambda)$ is positive, the exponent 
is an analytic function of $\lambda$, by Ruelle's theorem (Theorem 3.1. in \cite{Ruelle1979}).

\section{Uniquely ergodic weights: Proof of Theorem \ref{thm:UE}}
Let us borrow the notation and the argument from the above proof of Theorem \ref{thm:random}. Assume $f(x) = f(x_0, x_1, \cdots, x_{r-1})$. For fixed $\lambda\in \mathbb{R}$ and fixed $\omega\in \Omega$, let us define
a $q^{r-1}\times q^{r-1}$-matrix $$
A_\omega(\lambda) = (a_{u, v})_{(u,v)\in S^{r-1}\times S^{r-1}}$$
as follows: if the $(r-2)$-suffix of $u$ is equal to the 
$(r-2)$-prefix of $v$, i.e. $u=x_0x_1\cdots x_{r-2}$
and $v=x_1\cdots x_{r-1}$ for some $(x_0, x_1, \cdots, x_{r-1})\in S^r$, then 
$$
a_{u, v} = e^{\lambda \phi(\omega) f(x_0, x_1, \cdots, x_{r-1})};
$$  
otherwise $a_{u, v}=0$. $B_\omega$ is similarly defined as above.
Since the function $\omega \mapsto A_\omega$ is eventually positive (i.e. $B_\omega$ is strictly positive), 
by the part 3 of Theorem 3 from Furman \cite{Furman1997}, the following limit
exists
$$
\lim_{n\to\infty} \frac{1}{n} \log \|A_\omega(\lambda) A_{\Theta \omega}(\lambda)\cdots A_{\Theta^{n-1} \omega}(\lambda)\|
$$
for all $\omega$ (and all $\lambda$) and the limit is actually uniform in $\omega$.
The independence of $\omega$ of the limit is due to the ergodicity and the analyticity of the limit as function of 
$\lambda$ follows from Theorem \ref{thm:random}.

But notice that Theorem 3 in \cite{Furman1997} requires that $A_\omega$
belongs to $GL_{q^{r-1}}(\mathbb{R})$. It is not the case in general for our $A_\omega$. However the part 3 of Theorem 3 in \cite{Furman1997} doesn't need this condition. This is because
the entries of the positive matrices $B_{\Theta^n \omega}$ are bounded from below by a constant $\delta>0$ and from above by $\delta^{-1}$ ($\delta$ being independent of $\omega$ and of $n$). 
\medskip

Let us state Furman's result, that we have used above, by dropping the invertibility
of the matrix:
{\em Let $(X, \mu , T)$ be a uniquely ergodic system and suppose that $A$ is a continuous real $d\times d$-matrix function defined on $X$ and that there exists an integer $p\ge 1$ such that
	$$ A(T^{p-1} x)\cdots A(Tx) A(x) >0
	$$
meaning that all entries are positive for all $x$. Then for every $x\in X$ the following limit exists:
$$
    \lim_{n\to\infty} \frac{1}{n} \log \|A(T^{n-1} x)\cdots A(Tx) A(x) \|.
$$
The limit is actually uniform in $x \in X$.	
}

Lemma 5 in \cite{Furman1997} which was used in the proof of the above result can be modified  as follows without requiring the invertibility: {\em let $(B_n)$ be a sequence of positive $d\times d$-matrices with entries in the interval $[\delta, \delta^{-1}]$ ($\delta >1$ being a constant). Let 
$$
\Delta:=\left\{(x_i) \in \mathbb{R}^d: \sum_{i=1}^{d} x_i =1, x_i \ge 0\right\}
$$
and $\overline{\Delta}$ be the corresponding set of $\Delta$ in the projective space $P^{d-1}$. Then there exists a unique point $\overline{u}\in P^{d-1}$ such that 
$$
\bigcap_{n=1}^\infty \overline{B}_n \overline{\Delta}
=\{\overline{u}\}
$$
where $\overline{B}_n$ is the projective transformation associated to $B_n$. 
}
Here is a proof. Let 
$K$ be the cone $\{(x_1, \cdots, x_d): x_i\ge 0\  (1\le i\le d)\}$. The Hilbert projective metric defined in $\stackrel{\circ}{K}$ is equal to
$$
d(x, y) = \log \max_{ 1\le i \le d} \frac{x_i}{y_i} 
\max_{ 1\le i \le d}  \frac{y_i}{x_i}.
$$
See \cite{Bushell1973}.
Then for any positive matrix $B=(b_{i, j})$, we have
$$
d(Bx, By) = \log \max_{ 1\le i \le d} \frac{\sum_{j=1}^d b_{i, j}x_j}{ \sum_{j=1}^d b_{i, j} y_j} 
\max_{ 1\le i \le d}  \frac{\sum_{j=1}^d b_{i, j}y_j}{ \sum_{j=1}^d b_{i, j} x_j} .
$$
It easy to see that if $\delta \le b_{i, j}\le \delta^{-1}$
for all $i$ and $j$, we have
$$
d(Bx, By) \le 4 \log \frac{1}{\delta}<\infty.
$$
Thus, the hypothesis on $B_n$'s implies that the projective diameters of $B_n$'s are bounded, so that  the operators 
$\overline{B}_n$ are contractive with a uniform contracting ratio
$\tanh (\log \delta^{-1})<1$ (\cite{Birkhoff1957}, see also \cite{Bushell1973} p. 333).

\medskip

Based on Lemma 5 in \cite{Furman1997}, it is proved in \cite{Furman1997} that there exists a function $\overline{u}: X \to \overline{\Delta}$ such that 
$\overline{u}(Tx) = \overline{A}(x) \overline{u}(x)$ namely
$$A(x) \overline{u}(x) = \overline{u}(Tx) \|A(x) \overline{u}(x)\|.$$
Using this, we get
\begin{eqnarray*}
	A(n, x) \overline{u}(x)
	& = & A(T^{n-1}x) A(T^{n-2} x) \cdots A(Tx)  A(x)  \overline{u}(x)\\
	&= & A(T^{n-1}x) A(T^{n-2} x) \cdots A(Tx)  \overline{u}(Tx) \|A(x) \overline{u}(x)\|.
\end{eqnarray*} 
Inductively we get
$$
A(n, x) \overline{u}=
\overline{u}(T^n x)  \|A(T^{n-1}x) \overline{u}(T^{n-1}x)\| \cdots \|A(Tx) \overline{u}(Tx)\|  \|A(x) \overline{u}(x)\|.
$$
So,
$$
\frac{\log \|A(n, x) \overline{u}(x)\|}{n}
= \frac{\log \| \overline{u}(T^n x)\|}{n} + \frac{1}{n} \sum_{k=0}^{n-1} \phi(T^k x)
$$
where 
$
\phi (x) = \log \|A(x) \overline{u}(x)\|
$, which is a continuous function on $X$. Then we can conclude by the unique ergodicity of $T$. 

The above argument repeats that in \cite{Furman1997} with some modifications and details (it seems that there is something wrong   at the bottom of page 807 in \cite{Furman1997}). 

\section{Minimal and uniquely ergodic weights: Proof of Theorem \ref{thm:main2}} \label{sect:substitution}

We first recall some useful facts on orbital systems, especially orbital systems generated by primitive substitutive sequences and the notion of return word \cite{Durand1998}, which is the key
for proving Theorem \ref{thm:main2}. The reference \cite{Queffelec2010}
is a good source for primitive substitutive sequences.

\subsection{Minimal and uniquely ergodic sequences}
Let $\mathcal{A}$ be a finite set, called {\em alphabet}. Elements of $\mathcal{A}$ are called {\em letters}. A {\em word} of $\mathcal{A}$
is an element $x=x_0x_1\cdots x_{n-1}$ of  $\mathcal{A}^n$, where $n$
is 
denoted $|x|$ and is called the
 {\em length} of $x$. The length of the {\em empty-word} $\emptyset$
 is $0$. Let $\mathcal{A}^+ = \cup_{n=1}^\infty \mathcal{A}^n$
 and $\mathcal{A}^* =\mathcal{A}^+\cup\{\emptyset\}$. Sequences
 in $\mathcal{A}^{\mathbb{N}}$ are called {\em infinite words}.
 
  With concatenation, $\mathcal{A}^*$ becomes a monoid. Let $\mathcal{B}$ be another alphabet. By concatenation, every map
 $\varphi: \mathcal{A} \to \mathcal{B}^+$ induces a map 
 $\varphi: \mathcal{A}^* \to \mathcal{B}^*$, and a map from
 $\mathcal{A}^\mathbb{N}$ into $\mathcal{A}^\mathbb{N}$, which is still denoted by $\varphi$. 
 
 A {\em substitution} is a triple $(\zeta, \mathcal{A}, \alpha)$
 where $\mathcal{A}$ is an alphabet, $\zeta: \mathcal{A}\to \mathcal{A}^+$ is a map and $\alpha \in \mathcal{A}$, such that\\
 \indent (S1) \ the first letter of $\zeta(\alpha)$ is $\alpha$;\\
 \indent (S2) \ $\lim_{n\to\infty}|\zeta^n(\alpha)|=\infty$.\\
 The limit $u_\zeta:=\lim_{n\to\infty} \zeta^n(\alpha) \in \mathcal{A}^\mathbb{N}$ exists, and it is characterized by
 $\zeta(u_\zeta) =u_\zeta$ (i.e. $u_\zeta$ is a fixed point of $\zeta$) and the first letter of $u_\zeta$ is $\alpha$. If $\varphi: \mathcal{A} \to \mathcal{B}$ where $\mathcal{B}$ is another alphabet. We define $w_\zeta = \varphi(u_\zeta)$. Such a sequence is called a {\em substitutive sequence}.

 A substitution $(\zeta, \mathcal{A}, \alpha)$ is said to be {\em primitive} if there exits an integer $k$ such that 
 for all letters $\beta\in \mathcal{A}$ and $\gamma \in \mathcal{A}$,
 $\beta$ is a letter in $\zeta^k(\gamma)$. In this case, the corresponding sequence $w_\zeta$ is said to be {\em primitive}.
 
 Let $x=x_0x_1\cdots \in \mathcal{A}^\mathbb{N}$, let $m$ and $n$
 be two integers with $m\le n$. We write $x_{[m, n]}$ for the word 
 $x_m \cdots x_n$, called a {\em factor} of $x$. The index $m$ is called the {\em recurrence} of $x_{[m,n]}$.  Factors $x_{[0, n]}$
 are called {\em prefixes}. For a word $x=x_0x_1\cdots x_{\ell -1}$, we can also define its factors and prefixes. {\em  Suffixes} of this word
 $x$ are defined to be the words $x_{k} x_{k+1}\cdots x_{\ell- 1}$ for 
 $0\le k <\ell$.  
 
 An infinite sequence $x=x_0x_1\cdots \in \mathcal{A}^\mathbb{N}$ 
 is said to be {\em minimal} 
 if for every integer
 $\ell\ge 1$, there exists an integer $L\ge 1$ such that 
 each factor of $x$ of length $\ell$ occurs as factor of every factor
 of $x$ of length $L$. Let $\Theta: \mathcal{A}^\mathbb{N} \to \mathcal{A}^\mathbb{N}$ be the shift transformation. 
 That $x$ is minimal means that the orbital system $(\overline{\{\Theta^n x\}}, \Theta)$ is minimal. 
 If the system $(\overline{\{\Theta^n x\}}, \Theta)$
 is minimal and uniquely ergodic, we say $x$ is minimal and uniquely ergodic. 
 
 Let $x$ be a minimal sequence over an alphabet $\mathcal{A}$ and $u$ be a non-empty prefix of $x$. We call
 {\em return word over $u$} every factor $x_{[i, j-1]}$
 where $i$ and $j$ are two successive occurrences of $u$ in $x$.  
 We use $\mathcal{R}_u(x)$ to denote the set of all return words over $u$.

 Let $x$ be a minimal sequence.  For every prefix $u$ of $x$, the sequence has a unique decomposition
 \begin{equation}\label{eqn:mmm}
      x=m_0m_1m_2\cdots \in \mathcal{R}_u(x)^\mathbb{N};
 \end{equation}

 The following lemma contains the key facts for us.
 
 \begin{lemma}
 	\label{lem:RW} Suppose that $x \in \mathcal{A}^\infty$ is minimal and uniquely ergodic. 
 	For each prefix $u$ of 
 	 $x$, \\
 	 \indent {\rm (a)} the set of return word $\mathcal{R}_u(x)$ is finite.\\
 	 \indent {\rm (b)} the following frequencies exist:
 	 \begin{equation}
 	 p_v =\lim_{k\to\infty}\frac{\#\{0\le i< k: m_i =v\}}{k}
 	 \qquad (\forall v \in R_u(w))
 	 \end{equation}
 	 where $m_j$'s are the factors in the decomposition (\ref{eqn:mmm}) of $x$.
 	\end{lemma}
 	
Every primitive substitutive sequence is minimal and uniquely ergodic  (\cite{Michel1974}).

Let us look at the Thue-Morse sequence $(t_n)$ defined by the substitution $0\mapsto 01$, $1\mapsto 10$:
$$
   01 10 1001 10010110 1001 0110 1001011001101001 \cdots 
$$ 
If we take the prefix $u=0$, then we get the following decomposition 
$$
\underline{01 1} \ \underline{0 1} \ \underline{0} \ \underline{01 1}\
\underline{0} \ \underline{01}\ \underline{011} \ \underline{0 1}\
\underline{0}\ \underline{01}\ \underline{011}\ \underline{0 1}\ \underline{0}\ \underline{01}\ \underline{011}\ \underline{0}\ \underline{011}\ \underline{01}\ \underline{0}\ 01 \cdots 
$$
It is known that there is no cubes in $(t_n)$. It is easy to see that
$$
    \mathcal{R}_0((t_n)) =\{0, 01, 011\}.
$$
If we take the prefix $u=01$, then we get the following decomposition 
$$
\underline{01 1} \ \underline{0 10} \ \underline{01 1 0} \ \underline{01}\ \underline{011} \ \underline{0 1 0}\ \underline{01}\ \underline{011}\ \underline{0 10}\ \underline{01}\ \underline{0110}\ \underline{011}\ \underline{010}\ 01 \cdots 
$$
In this case we have
$$
     \mathcal{R}_{01}((t_n))=\{01, 010, 011, 0110\}.
$$

\subsection{Proof of Theorem \ref{thm:main2}}
The assumption (H2) is easy to check by using the hypothesis of the bounded variation of $f$. In the following, we check the assumption (H1).

If $w$ is periodic, then $\phi$ is well defined. So, in the following,
we assume that $w$ is not periodic.
 
Let $n\ge 1$ be a fix integer and let $u$ be the prefix of $w$ having length $n$. Since $w$ is aperiodic, so is $u$. Therefore every
return word $v\in R_u(w)$ has length $|v|\ge \frac{n}{2}$. Assume that
$$
     w=v_1v_2\cdots v_k\cdots  \quad {\rm with} \ \ v_j \in R_u(w).
$$
Such a decomposition exists and is unique, see (\ref{eqn:mmm}).
For any word $b=b_0b_1\cdots b_{m-1} \in \mathcal{A}^+$, we introduce the notation
$$
       Z_b:= Z_{b}(\lambda) = \int \exp\left(\lambda \sum_{j=0}^{m-1} b_j f(T^j x)\right) dx.
$$
Notice that 
\begin{equation}\label{Z_Z}
Z_{m, n}(\lambda) = Z_{w_m w_{m+1} \cdots w_{m+n-1}}(\lambda).
\end{equation}
Indeed, by the definition of $Z_{m, n}(\lambda)$ and the invariance of $dx$, we have
\begin{eqnarray*}
      Z_{m, n}(\lambda) &=&     \int \exp\left( \lambda \sum_{j=m}^{m+n-1} w_j f(T^j x)\right) dx  \\
       &=&   \int \exp\left( \lambda \sum_{j=m}^{m+n-1} w_j f(T^{j-m} x)\right) dx. 
\end{eqnarray*}
Thus, by Lemma (\ref{lem:ZZZ}), we have
$$
      C^{-k} \le \frac{Z_{v_1v_2\cdots v_k}}{Z_{v_1}Z_{v_2}\cdots Z_{v_k}}\le C^k.
$$
It follows that
\begin{equation}\label{eqn:Z_v}
   \frac{\log Z_{v_1v_2\cdots v_k}}{|v_1v_2\cdots v_k|}  = 
   \frac{k}{|v_1v_2\cdots v_k|}\cdot \frac{1}{k}\sum_{i=1}^k \log Z_{v_i} \pm \frac{k}{|v_1v_2\cdots v_k|} \log C.
\end{equation}
 By Lemma \ref{lem:RW} (b), the following frequencies exist: 
$$
      p_v :=\lim_{k\to\infty}\frac{\#\{1\le i\le k: v_i =v\}}{k}
      \qquad (\forall v \in R_u(w)).
$$
Therefore, by  Lemma \ref{lem:RW} (a), 
$$
    \lim_{k\to \infty}  \frac{k}{|v_1v_2\cdots v_k|} = \frac{1}{\sum_{v\in R_u(w)} p_v |v|}.
$$
$$
\lim_{k\to \infty}  \frac{1}{k}\sum_{i=1}^k \log Z_{v_i}  = \sum_{v\in R_u(w)} p_v \log Z_v.
$$
Let
$$
\underline{L}_u =  \liminf_k\frac{\log Z_{v_1v_2\cdots v_k}}{|v_1v_2\cdots v_k|}  ,\qquad
 \overline{L}_u =  \limsup_k  \frac{\log Z_{v_1v_2\cdots v_k}}{|v_1v_2\cdots v_k|}; 
$$
and
$$
       A_u = \frac{\sum_{v\in R_u(w)} p_v \log Z_v}{\sum_{v\in R_u(w)} p_v |v|}.
$$
Notice that $\frac{k}{|v_1v_2\cdots v_k|}\le \frac{2}{n}$. From
(\ref{eqn:Z_v}), we get
\begin{equation}\label{eq:L-A}
      A_u - \frac{2\log C}{n}  \le   \underline{L}_u \le \overline{L}_u \le A_u + \frac{2\log C}{n}. 
\end{equation}
Then 
\begin{equation}\label{eq:L}
   0\le   \overline{L}_u  - \underline{L}_u \le \frac{4\log C}{n}.
\end{equation}

For any $N$, there exists a unique integer $k$ such that
$$
    |v_1v_2\cdots v_k|\le N < |v_1v_2 \cdots v_k v_{k+1}|.
$$
So, by the definition of $Z_N(\lambda)$, we have 
$$
    \log Z_N(\lambda) = \log Z_{v_1v_2\cdots v_k} \ \pm \ |\lambda|M \|w\|_\infty\|f\|_\infty
$$
where $M=\max_{v\in \mathcal{R}_u(w)} |v|$. It follows that 
\begin{equation}\label{eq:phi-L}
     \underline{\phi} =  \underline{L}_u, \quad 
      \overline{\phi}=\overline{L}_u,
\end{equation}
where
$$
    \underline{\phi}(\lambda) = \liminf_{N\to \infty} \frac{\log Z_N(\lambda)}{N}, \quad \overline{\phi}(\lambda) = \limsup_{N\to \infty} \frac{\log Z_N(\lambda)}{N}.
$$
From (\ref{eq:L}) and (\ref{eq:phi-L}), we get
$$
0\le   \overline{\phi}(\lambda)  - \underline{\phi}(\lambda) \le \frac{4\log C}{n}.
$$
Observe that both $\underline{\phi}$ and $\overline{\phi}$
are independent of $n$. 
Letting $n\to \infty$, we get $\underline{\phi}(\lambda) = \overline{\phi}(\lambda)$. The theorem is thus proved. 
\medskip
 
Notice that from (\ref{eq:L-A}) and (\ref{eq:phi-L}), we get the following approximation of $\phi$ by the real analytic functions $A_u$:
\begin{equation}\label{eq:phi-A}
    \phi(\lambda) = \frac{\sum_{v\in R_u(w)} p_v \log Z_v(\lambda)}{\sum_{v\in R_u(w)} p_v |v|} \pm \frac{2\log C}{n}.
\end{equation}
Recall that $n$ is the length of the prefix $u$. This approximation is uniform on any compact set of $\lambda$ because $C= e^{O(\lambda)}$.

\section{Proof of Theorem \ref{thm1} }

The condition (C1) in Theorem \ref{thm1} implies the condition (H1) in Theorem \ref{thm:main1}. 
So, in order to apply Theorem \ref{thm:main1} to prove  Theorem \ref{thm1}, it suffices to compute the pressure function.

\subsection{Computation of the pressure function}
The  observation stated in the following lemma will allow us to compute 
the pressure function. 

\begin{lemma}[Bernoullicity] \label{lem1} Let 
	$$
	f_n(x) = x_n g_n(x_{n+1}, \cdots, x_{n+p}, \cdots)
	$$
	where $g_n$ is a 
	Borel function taking values in $\{-1, 1\}$. Then $f_n$'s are independent symmetric Bernoulli variables, in other words
		$$
		\forall (t_1, \cdots, t_k)\in \mathbb{R}^k, \quad 
	\mathbb{E} e^{t_1 f_1 + \cdots + t_n f_n} = \frac{1}{2^n} 
	\prod_{k=1}^n (e^{t_k} + e^{-t_k}).	     
	$$
	\end{lemma}

\begin{proof} First remark that for any $r\in \mathbb{R}$, we have 
	$\mathbb{E} e^{x_j r} = \frac{1}{2} (e^{r} + e^{-r})$, which depends only on the absolute value of $r$. Write
	$$
	    \mathbb{E} e^{t_1 f_1 + \cdots + t_n f_n}
	        = \mathbb{E}[ e^{t_2 f_2 + \cdots + t_n f_n} \mathbb{E} \left(e^{t_1 f_1} | x_2, \cdots, x_n, \cdots\right)].
	$$
	Observe that $x_1$ is independent of $x_2, x_3, \cdots$. By the above remark we have 
	$$ \mathbb{E} \left(e^{t_1 f_1} | x_2, \cdots, x_n, \cdots\right) = \frac{1}{2} (e^{t_1} + e^{-t_1}),$$
	because the conditional expectation is equal to the expectation with respect to $x_1$ with $x_2, x_3, \cdots$ being fixed. 
	Thus, by induction, we get
	$$
	     \mathbb{E} e^{t_1 f_1 + \cdots + t_n f_n} = \frac{1}{2^n} 
	     \prod_{k=1}^n (e^{t_k} + e^{-t_k}).	     
	$$
	\end{proof}

It follows that the pressure function $\phi$ is independent of the form of the functions $g_n$'s. So, in the following, without of loss of generality we continue our discussion
with $g_n(x_{n+1}, x_{n+2}, \cdots) =x_{n+1}$. According to 
Theorem \ref{thm:main1}, the result  on $\dim H(\alpha)$ depends only on the function
$$
   \phi(\lambda) = 
   \lim_{N\to\infty} \frac{1} {N} \log \mathbb{E} e^{\lambda \sum_{n=1}^N w_n x_nx_{n+1}}
$$
provided that the limit exists. The limit does exist and is computable.


\begin{lemma}[Pressure function]\label{lem:pressure} Suppose that $f_n$'s
	satisfy the assumption made in Lemma \ref{lem1} and that 
	 $(w_n)$ take values $v_0, v_1, \cdots, v_m$ such that the frequencies $p_j$'s defined by (\ref{eqn:freq}) exist.
Then
\begin{equation}\label{PF}
      \phi(\lambda)
       =\sum_{j=0}^m p_j \log (e^{\lambda v_j} + e^{-\lambda v_j}) -  \log 2.
\end{equation}
\end{lemma}

\begin{proof} Lemma \ref{lem1} gives 

	\begin{equation}\label{PF2} \mathbb{E} e^{\lambda \sum_{n=1}^N w_n x_nx_{n+1}} = 2^{-N} \prod_{n=1}^N (e^{\lambda w_n} + e^{-\lambda w_n}). 
\end{equation}
Then (\ref{PF}) follows immediately if we use the hypothesis on $(w_n)$.
\medskip

We would like to give another proof of (\ref{PF2}). One reason is to get rid of Lemma \ref{lem1} which could be mysterious for some readers. The other reason is that this method  will allow us to treat other cases. 

Since the function $\sum_{n=1}^N w_nx_n x_{n+1}$ is constant on cylinders of length $N+1$, we have
$$
\mathbb{E} e^{\lambda \sum_{n=1}^N w_n x_nx_{n+1}}
= \frac{1}{2^{N+1}} \sum_{x_1, \cdots, x_{N+1} \in \{-1, 1\}}
e^{\lambda \sum_{n=1}^N w_n x_nx_{n+1}}.
$$
Let $$
A_n = 
\begin{pmatrix}   e^{\lambda w_n}  &  e^{-\lambda w_n}\\
e^{-\lambda w_n}  &  e^{\lambda w_n}          
\end{pmatrix}.
$$
It is clear that
\begin{equation}\label{comput1}
\mathbb{E} e^{\lambda \sum_{n=1}^N w_n x_nx_{n+1}}
= \frac{1}{2^{N+1}}  \|A_1\cdots A_N\|_S
\end{equation}
where $\|B\|_S$ denotes the norm of a matrix $B$,  the sum of all elements of  $B$. 
Notice that all $A_n$'s are of the form
$$
A= \begin{pmatrix}   a  &b \\
b &    a             
\end{pmatrix}     \quad  (a,b \in \mathbb{R})
$$
which commute each other. 
Indeed, they can be simultaneously diagonalized as follows
\begin{equation}\label{WF}
R^{-1} A R = D:=\begin{pmatrix}   a+b  &0 \\
0 & a-b           
\end{pmatrix},
\end{equation}
where 
$$
R = \frac{\sqrt{2}}{2}\begin{pmatrix}   1  & -1\\
1&    1\\
\end{pmatrix}, \quad  
R^{-1} = \frac{\sqrt{2}}{2}\begin{pmatrix}   1  & 1\\
-1&    1\\
\end{pmatrix},
$$
which are independent of $a$ and $b$. 
Apply (\ref{WF}) to $A=A_n$ ($1\le n \le N$) to get $R^{-1}A_n R = D_n$ with
$$
D_n  = 
\begin{pmatrix}   e^{\lambda w_n}+ e^{-\lambda w_n}  &\\
& e^{\lambda w_n} -e^{-\lambda w_n}           
\end{pmatrix}.
$$
Then $A_1\cdots A_n = R D_1\cdots D_N R^{-1}$. Notice that
$$
R  \begin{pmatrix}   u  & \\
& v           
\end{pmatrix} R^{-1} =\frac{1}{2}
\begin{pmatrix}   u+v  & u-v \\
u-v & u+v           
\end{pmatrix}
$$
and the sum of entries of the last matrix equals to $2u$.
Then 
$$
\|A_1\cdots A_n\|_S = 2 \prod_{n=1}^N (e^{\lambda w_n} + e^{-\lambda w_n}).
$$
This, together with (\ref{comput1}), leads to (\ref{PF2}). 
\end{proof}

\subsection{Proof of Theorem \ref{thm1}}
By Lemma \ref{lem:pressure}, we have
$$\psi(\lambda)= \sum_{j=0}^m p_j \log (e^{\lambda v_j} + e^{-\lambda v_j}).
$$
According to Theorem \ref{thm:main1}, we have to compute the conjugate function $\psi^*$. By simple calculations we get that 
$$
\psi'(\lambda)
= \sum_{j=0}^m p_j v_j \frac{e^{\lambda v_j}- e^{-\lambda v_j}}{e^{\lambda v_j}+e^{-\lambda v_j}} = \sum_{j=0}^m p_jv_j \frac{e^{2 \lambda v_j}- 1}{e^{2\lambda v_j}+1} = \sum_{j=0}^m p_j v_j - 2 \sum_{j=0}^m \frac{ p_j v_j}{e^{2\lambda v_j}+1}.
$$
$$
\phi'(+\infty) = \sum_{j=0}^m p_j |v_j|, \quad \phi'(-\infty) = - \sum_{j=0}^m p_j |v_j|.
$$
$$
\psi''(\lambda)
=  4 \sum_{j=0}^m  \frac{p_jv_j^2 e^{2 \lambda v_j}}{(e^{2\lambda v_j}+1)^2} >0.
$$
Then, for any $\alpha \in (\psi'(-\infty), \psi'(+\infty))$, there exists a unique
$\lambda_\alpha$ such that $\psi'(\lambda_\alpha) =\alpha$, i.e.
$$
\sum_{j=0}^m p_j v_j \frac{e^{\lambda_\alpha v_j}- e^{-\lambda_\alpha v_j}}{e^{\lambda_\alpha v_j}+e^{-\lambda_\alpha v_j}} = \alpha.
$$
So, $\psi^*(\alpha) =  \alpha \lambda_\alpha - \psi(\lambda_\alpha) $, which gives the formula
$$
\dim E(\alpha) =  
\frac{1}{\log 2}\sum_{j=0}^m p_j \left(\log (e^{\lambda_\alpha v_j} + e^{-\lambda_\alpha v_j})
-  \lambda_\alpha v_j \frac{e^{\lambda_\alpha v_j}- e^{-\lambda_\alpha v_j}}{e^{\lambda_\alpha v_j}+e^{-\lambda_\alpha v_j}}\right).
$$

\subsection{Gibbs measures are Markovian measures}
In the case $$f_n(x_n, x_{n+1}) = w_n x_n x_{n+1},$$
we can directly prove Theorem \ref{thm1} without using Theorem
 \ref{thm:main1}. Because 
we can directly construct the Gibbs measures as  inhomogeneous Markov measures and compute their dimensions without using Lemma \ref{lem:dim}. 

Consider the stochastic matrix  $$
P_n = 
\frac{1}{e^{\lambda w_n}+e^{-\lambda w_n}} 
\begin{pmatrix}   e^{\lambda w_n} &  e^{-\lambda w_n}\\
e^{-\lambda w_n}  &  e^{\lambda w_n}          
\end{pmatrix}.
$$
We denote it by $(p_{i, j}^{(n)})$. It is clear
that $(\frac{1}{2}, \frac{1}{2})$ is a left invariant probability vector of all
$P_n$. Let us define the inhomogeneous Markov measure $\mu_{\lambda}$ by
$$
   \mu_{\lambda}([x_0x_1\cdots x_n])= \frac{1}{2}p^{(0)}_{x_0,x_1}\cdots p^{(n-1)}_{x_{n-1},x_n}.
$$
In other words,
\begin{equation}\label{def:markov}
\mu_{\lambda}([x_0 x_1\cdots x_n])= \frac{1}{2 Z_n(\lambda)} 
\exp \left( \lambda \sum_{k=0}^{n-1} w_kx_k x_{k+1}\right).
\end{equation}
where $Z_n(\lambda) = \prod_{k=0}^{n-1}(e^{\lambda w_k} + e^{-\lambda w_k})$.

\begin{lemma}[Law of large numbers] \label{LLN3}For every $n\ge 0$ we have
	   $$
	   \int x_{n} x_{n+1} d\mu_{\lambda}(x) = \frac{e^{\lambda w_n}- e^{-\lambda w_n}}{e^{\lambda w_n}+e^{-\lambda w_n}}.$$
	   For $\mu_{\lambda}$-almost all $x$ we have
	   \begin{equation}\label{LLN}
	      \lim_{N\to\infty} \frac{1}{N}\sum_{n=0}^{N-1} w_n x_n x_{n+1}
	      = \sum_{j=0}^m p_jv_j \frac{e^{\lambda v_j}- e^{-\lambda v_j}}{e^{\lambda v_j}+e^{-\lambda v_j}}.
	   \end{equation}
	\end{lemma}
\begin{proof} By the definition of $\mu_{\lambda}$ we have 
	   $$\int x_{n} x_{n+1} d\mu_{\lambda}(x) 
	   = \sum_{x_0, x_1, \cdots, x_{n+1} } x_{n} x_{n+1}\cdot \frac{1}{2} p^{(0)}_{x_0, x_1} \cdots p^{(n)}_{x_n, x_{n+1}}.$$
	   Since $\sum_i p^{(k)}_{i, j} =1$, we have
	   $$\int x_{n} x_{n+1} d\mu_{\lambda}(x) 
	   =  \frac{1}{2}  \sum_{x_{n}, x_{n+1} } x_{n} x_{n+1} p^{(n)}_{x_{n}, x_{n+1}} = \frac{e^{\lambda w_n}- e^{-\lambda w_n}}{e^{\lambda w_n}+e^{-\lambda w_n}}$$
	   where the last equality follows from the definition of $P_n$.
	 	
	 	Let $Y_n = x_nx_{n+1}$. Similar computation shows that $Y_n - \mathbb{E}_{\mu_\lambda} Y_n$
	 	are orthogonal. Then (\ref{LLN}) follows from the Menshov theorem (see \cite{KSZ1948}) and the Kronecker lemma (see \cite{Shiryaev1996}).
	\end{proof}

\begin{lemma}[Dimensions of Markov-Gibbs measures]
	$$\dim \mu_{\lambda} = \frac{1}{\log 2}\sum_{j=0}^m p_j \left(\log (e^{\lambda v_j} + e^{-\lambda v_j})
	-  \lambda v_j \frac{e^{\lambda v_j}- e^{-\lambda v_j}}{e^{\lambda v_j}+e^{-\lambda v_j}}\right).
	$$ 
	\begin{proof} From the definition (\ref{def:markov}) of $\mu_{ \lambda}$, we have
		$$
		  \frac{\log \mu_{\lambda}([x_0 x_1\cdots x_n])}{\log 2^{-n}}
		  = \frac{\log_2 Z_n(\lambda)}{n}    - \frac{\lambda}{\log 2 \cdot n}\sum_{k=0}^{n-1} w_k x_{k}x_{k+1}+ o(1)
		$$
		Then by Lemma \ref{LLN3}, $\mu_{\lambda}$-a.e. we have
			$$
			D(\mu_{\lambda}, x)
			= \frac{1}{\log 2}\sum_{j=0}^m p_j \log (e^{\lambda v_j} + e^{-\lambda v_j})
			-  \frac{\lambda}{\log 2} \sum_{j=0}^m p_jv_j \frac{e^{\lambda v_j}- e^{-\lambda v_j}}{e^{\lambda v_j}+e^{-\lambda v_j}}.
			$$
		\end{proof}
\end{lemma}

\section{Proof of Theorem \ref{thm2}}
In the case of M\"{o}bius weights $w_n = \mu(n)$, by Lemma \ref{lem:pressure} we have
$$
   \psi(\lambda) = f_0\log 2  + (1-f_0) \log (e^{\lambda} + e^{-\lambda})
$$
where $1-f_0 = \frac{6}{\pi^2}$, because it is well known that
$$
  \lim_{N\to\infty}\frac{1}{N}\sum_{n=1}^{N}\mu(n) =0, \quad
  \lim_{N\to\infty}\frac{1}{N}\sum_{n=1}^{N}|\mu(n)| =\frac{6}{\pi^2}
$$
(see \cite{Tenenbaum1995}, Theorem 3.8 and Theorem 3.10).
 For $\alpha \in (-6/\pi^2, 6/\pi^2)$
we can solve the equation $\psi'(\lambda) =\alpha$, i.e.
$$
     \frac{6}{\pi^2}\frac{e^{\lambda} - e^{-\lambda}}{e^{\lambda} + e^{-\lambda}}=\alpha.
$$
Indeed, let $\alpha' = \frac{\alpha}{6/\pi^2} \in (0,1)$. Then we get the solution $\lambda_\alpha$:
$$
     e^{\lambda_\alpha} = \sqrt{\frac{1+\alpha'}{1-\alpha'}},
     \quad i.e. \ \ \lambda_\alpha = \frac{1}{2} \log \frac{1+\alpha'}{1-\alpha'}.
$$
Then 
\begin{eqnarray*}	    
	- \psi^*(\alpha) &=& \psi(\lambda_\alpha) - \alpha \lambda_\alpha \\ 
	 &=&  f_0\log 2 + (1-f_0) \log \left( \sqrt{\frac{1+\alpha'}{1-\alpha'}} + \sqrt{\frac{1-\alpha'}{1+\alpha'}} \right) - \frac{\alpha}{2} \log \frac{1+\alpha'}{1-\alpha'}.
	\end{eqnarray*}
Notice that 
\begin{eqnarray*}
    2 \log \left( \sqrt{\frac{1+\alpha'}{1-\alpha'}} + \sqrt{\frac{1-\alpha'}{1+\alpha'}} \right)
    & = & \log \left( \frac{1+\alpha'}{1-\alpha '} +  \frac{1-\alpha'}{1+\alpha '} +2 \right)\\
    &= & \log \frac{4}{(1+\alpha')(1-\alpha')}.
\end{eqnarray*}
So, letting $p=(1+\alpha')/2$ and $p' =(1-\alpha')/2$, we get
\begin{eqnarray*}	    
	- \psi^*(\alpha)  
	&=&  f_0\log 2 + \frac{1-f_0}{2} \log \frac{4}{(1+\alpha')(1-\alpha')}
	 - (1-f_0) \frac{\alpha'}{2} \log \frac{1+\alpha'}{1-\alpha'}\\
	 &=& f_0 \log 2 - (1-f_0)\frac{1}{2} \log (p p') - (1-f_0) \frac{\alpha'}{2} \log (p/p') \\
	 & = & f_0 \log 2 - (1-f_0)\left( \frac{1+\alpha}{2} \log p  +  \frac{1-\alpha'}{2} \log p' \right)\\
	 & = & f_0 \log 2 +(1-f_0) H(p). 
\end{eqnarray*}
So, by Theorem \ref{thm1}, we get 
$$
	    \dim F(\alpha) = 1- \frac{6}{\pi^2} + 
	    \frac{6}{\pi^2 \log 2} H\left( \frac{1}{2} + \frac{ \pi^2}{12} \alpha \right).
	$$

The above proof, without any change, actually proves the the following more general result.

\begin{theorem} 
Assume that 
 $(w_n)$ is a sequence  taking $-1, 0, 1$ as values and having $f_0$ as the frequency of $0$'s. For any 
$\alpha \in (-(1-f_0), 1-f_0)$, we have
$$
\dim F(\alpha) = f_0 + 
\frac{1-f_0}{\log 2} H\left( \frac{1}{2} + \frac{ \alpha}{2 (1-f_0)}  \right),
$$
where
$$
  F(\alpha) = \left\{x \in \{-1,1\}^\mathbb{N}: \lim_{N\to\infty} \frac{1}{N}\sum_{n=1}^N w_n x_nx_{n+1}=\alpha \right\}.
$$
\end{theorem}

\bigskip

\section{Final remarks}

R.1. We can consider complex valued or vector valued functions $f_n$. Assume that $f_n$'s take values in $\mathbb{R}^d$. Then we have to change the definition of $Z_n(\lambda)$ as follows 
$$
Z_n(\lambda) =\mathbb{E} \exp \left(\lambda \cdot \sum_{k=0}^{n-1}f_k(x_k, x_{k+1}, \dotsc)\right), 
\quad (\lambda \in \mathbb{R}^d)
$$  
where $\lambda \cdot a$ denotes the inner product in $\mathbb{R}^d$.

R.2. Theorem \ref{thm:main1} is not applicable to the case
$$f_n(x_n, x_{n+1},\cdots) =w_n x_n x_{2n}.$$ 
Because, although the condition 
(H1) is satisfied (see Lemma \ref{lem:pressure}), but the condition (H2) is not satisfied. 
New ideas are needed to study this case.
The special case where $w_n=1$ for all $n$ was treated in \cite{FSW,PS2011}. The more general case $f_n(x) = f(x_n, x_{2n})$ was studied in \cite{FSW} and the method of \cite{FSW} could be used to treat a general $(w_n)$ which takes a finite number of values and admits frequencies for all possible values.  See \cite{FLW2018,LR2013,PSSS2014, Wu-thesis} for related works.
A form of non-linear thermodynamic formalism based on solutions to a nonlinear
equation was useful for such a problem \cite{KPS2012,FSW2011,FSW}.
The idea comes from Kenyon, Peres and Solomyak \cite{KPS2012}.
Pollicott \cite{Pollicott2017} considered a more general setting of nonlinear transfer operator. For a nonlinear Perron-Frobenius theory,
see \cite{LN2012}. 
\medskip

R.3. 
If $f_n(x_n, x_{n+1},\cdots)$ is of the form $w_n f(x_n, x_{n+1})$
where $f: S \times S \to \mathbb{R}$ is an arbitrary function, we can apply Theorem \ref{thm:main1} to this case. But we have to make sure that the limit defining $\phi$ exists. Better is to ensure the differentiability of $\phi$. However both are questionable  and works are to be done for a given weight $(w_n)$, except for the dynamically produced weights considered in Theorem \ref{thm:random},  Theorem \ref{thm:UE} and Theorem \ref{thm:main2}. 
\medskip

{\bf Problem 1.} Find conditions on $(w_n)$ and on $f$ such that 
$\phi$ is well defined and differentiable.  

In the following, we discuss some sub-problems.  
\medskip

R.4. A very special case of Problem 1 is as follows.

{\bf Problem 2.} Suppose that $w_n$ is the M\"{o}bius function $\mu(n)$
and $f: \{0,1\}\times \{0,1\} \to \{0,1\}$ is defined by
$f(x, y) = xy$. Is $\phi$ well defined and differentiable ?
We emphasize that it is $\{0,1\}$ but not $\{-1, 1\}$. If we would like to work with $\{-1, 1\}$, the problem arises for 
$$
      f(x, y) = a xy +bx+cy     \quad (a, b,c \ {\rm being\ \ constants}).
$$

Here is an idea to attack such a problem, which was used for proving Theorem \ref{thm:random} and
Lemma \ref{lem:pressure}. Assume $f_n(x_n, x_{n+1},\cdots)=w_n f(x_n, x_{n+1})$
where $f: S \times S \to \mathbb{R}$ and $(w_n)$ are given. For any
$\lambda\in \mathbb{R}$, define a $S\times S$-matrix
	 $$
	A_n: = A_n(\lambda): = \Big( e^{\lambda f(i, j)}\Big)_{(i, j)\in S\times S}.
	$$
	By the same argument as in the proof of (\ref{comput1}), we can obtain 
	\begin{equation} 
	Z_n (\lambda) =\mathbb{E} e^{\lambda \sum_{n=1}^N w_n f(x_n,x_{n+1})}
	= \frac{1}{q^{N+1}}  \|A_1\cdots A_N\|
	\end{equation}
	where $\|B\|$ denotes the sum of all elements of a matrix $B$. 
	So, we are led to prove the existence of the following 
	Liapounov exponent
	\begin{equation}\label{eqn:Liap}
	   L(\lambda) = \lim_{N\to\infty}\frac{1}{N} \log \|A_1\cdots A_N\|.
	\end{equation}
	Notice that $L(\lambda)$ is nothing but $\psi(\lambda)$, if the
	limit in (\ref{eqn:Liap}) exists. 
	
	Remark that for the case concerned by Problem 2, the matrix 
	$A_n(\lambda)$ takes a simple form
	$$
	A_n(\lambda)= \begin{pmatrix}   1  &1 \\
	1 &    e^{\lambda \, \mu(n) }             
	\end{pmatrix}.    
	$$
	
\bigskip
	
R.5. 	Keep the same notation as in R.4. Replace $(w_n)$ by   is a sequence of independent and identically distributed random variables $(\omega_n)$ taking a finite number of values.  
By Theorem 
\ref{thm:random}, $L(\lambda)$ is analytic.
	The method presented by Pollicott in  \cite{Pollicott2010} can be used to numerically compute $L(\lambda)$.
		\medskip
	
R.6. 	If $(w_n)$ is a primitive substitutive sequence, we have proved that $L(\lambda)$ is well defined and is analytic (Theorem \ref{thm:UE}). 
	\medskip
	
		{\bf Problem 3. } Suppose that $(w_n)$ is a primitive substitutive sequence. Then $L(\lambda)$ differentiable, by Theorem \ref{thm:UE}.
		Is it possible to get a closed form for $L(\lambda)$ ?
		How about Thue-Morse sequence or other specific sequences?
		
		\medskip
	
R.7. Following Katznelson and Weiss \cite{KW1982}, Furman (\cite{Furman1997}, Theorem 1) proved that on any unique ergodic system $(\Omega, \Theta)$, there exist continuous subadditive cocyles $(f_n)$ such that
the limit of $n^{-1}f_n(\omega)$ doesn't exist for some $\omega$. Our pressures considered in Theorem \ref{thm:UE} 
are defined by the limit for  special cycles. By Theorem \ref{thm:UE}, the limit defining the pressure does exist under the condition that $f$ depends only on a finite number of coordinates. 

	{\bf Problem 4. }
Can we drop this condition of dependence on finite coordinates but we assume that $f$
is of bounded variation ? 
\bigskip

Finally, let us repeat that we are interested in evaluating
$$
   \lim_{n \to\infty} \frac{1}{n} \log \left\|
   \begin{pmatrix}   1  &1 \\
   1 &    e^{\lambda \, w_1 }
   \end{pmatrix} 
   \begin{pmatrix}   1  &1 \\
   1 &    e^{\lambda \, w_2 } 
   \end{pmatrix} 
   \cdots \begin{pmatrix}   1  &1 \\
   1 &    e^{\lambda \, w_n } 
   \end{pmatrix} 
   \right\|
$$    
for different weights $(w_n)$. Three questions are associated:
does the limit exist ? is the limit differentiable as function of $\lambda$ ? is it possible to compute the limit ?

\bigskip

{\em Addendum} B. B\'ar\'any, M. Rams and  R. X.  Shi have obtained some results similar to  Theorem \ref{thm:random} and Theorem  \ref{thm1} with a different approach, which will be presented in a forthcoming paper.

\end{document}